\documentclass[a4paper, 11pt]{article}

\usepackage{authblk}
\usepackage{setspace}
\usepackage{comment}
\usepackage{amssymb}
\usepackage{amsthm}
\usepackage{amscd}
\usepackage{graphicx}
\usepackage{titlesec}
\usepackage{indentfirst}
\usepackage{cite}
\usepackage{xcolor}
\titleformat{\section}{\bfseries}{\thesection .}{0.5em}{}
\titleformat{\subsection}{\itshape}{\thesubsection .}{0.5em}{}
\usepackage[top=25.4mm, bottom=25.4mm, left=27.7mm, right=27.8mm]{geometry}

\usepackage[tbtags]{amsmath}
\usepackage{booktabs,cite}
\allowdisplaybreaks  
\usepackage{mathrsfs}
\usepackage{multirow}
\usepackage{caption}
\begin{document}

\newtheorem{theorem}{Theorem}[section]  
\newtheorem{example}{Example}[section]            
\newtheorem{algorithm}{Algorithm}[section]

\newtheorem{axiom}{Axiom}[section]
\newtheorem{property}{Property}[section]

\newtheorem{proposition}[theorem]{Proposition}%
\newtheorem{corollary}[theorem]{Corollary}%
\newtheorem{definition}[theorem]{Definition}%
\newtheorem{lemma}{Lemma}[section]
\newtheorem{cor}{Corollary}[section]
\newtheorem{remark}{Remark}[section]
\newtheorem{condition}{Condition}[section]
\newtheorem{conclusion}{Conclusion}[section]
\newtheorem{assumption}{Assumption}[section]

\title{$C$-existence families, $C$-semigroups and their associated abstract Cauchy problems in complete random normed modules}
\author{
{\normalsize  Xia Zhang, Leilei Wei and Ming Liu\thanks{
\textbf{CONTACT} Ming Liu; E-mail: liuming@tiangong.edu.cn.} }\\

\textit{\small School of Mathematical Sciences, Tiangong University, Tianjin 300387, P.R.China}}




\date{}
\maketitle

\renewcommand{\baselinestretch}{1.2}
\large\normalsize
\noindent \rule[0.5pt]{14.5cm}{0.6pt}\\
\noindent
\textbf{Abstract}\\
 In this paper, we first introduce the notion of a (mild) $C$-existence family in complete random  normed modules, then we prove that a (mild) $C$-existence family can guarantee the existence of the (mild) solutions of the associated abstract Cauchy problem in the random setting. Second, we investigate several important properties peculiar to locally almost surely  bounded $C$-semigroups in complete
random  normed modules, which are not involved in the classical theory of $C$-semigroups. Finally, based on the above work, some relations among $C$-existence families, $C$-semigroups and their associated abstract Cauchy problems in complete random  normed modules are established, which extend and improve some known results. Besides,
an application to a type of stochastic differential equations is also given.

\vspace{0.3cm}

\noindent \rule[0.5pt]{14.5cm}{0.6pt}\\
\noindent
\textbf{Mathematics Subject Classification (2021 MSC):} 46B09, 46H25, 60H20
\section{Introduction}\label{sec1}
\label{intro}
Suppose that $X$ is a Banach space and $A$ is a closed linear operator on $X$, then the abstract Cauchy problem on $X$ is described as the following form
$$~~~~~~~~~~~~~~~~~~~~~~~~~~~~~~~~~~~~~~~~~~
\left\{\begin{array}{l}
\frac{d z(t)}{d t}=A z(t),~~ \forall t\geq 0, \\
z(0)=z_{0}.
\end{array} \qquad  \quad~~~~~~~~~~~~~~~~~~~~~~~~~~~~~~~~~~~~~~~~~~~~~~~~~~~~~~~~~~~~~~~~~~~~~~~~~~~~~~~~~~~~~~~~~~~~~~~~~~~~~~~~~~~~~~~~~~~~~~~~~~~~~~~~~~~~~~(*)\right.
$$
It is well known that  $A$ generating a $C_{0}$-semigroup corresponds to the abstract Cauchy problem ($\ast$) having a unique mild solution, for all initial data $z_{0}$ \cite{Engel2000,Pazy}.
What should be done when the abstract Cauchy problem ($\ast$) is not well-posed or does not have a mild solution for all initial data? Ill-posed problems arise naturally (see \cite{Pay}).
There are two ways one can choose. One can look for initial data in the original space that yield solutions. Or one can renorm a subspace in such an approach that $A$, when restricted to that space, generates a $C_{0}$-semigroup.
Just as stated in \cite{deL1}, $C$-existence families and $C$-semigroups have provided a simple and powerful tool for either way, and the studies on such families of operators have also achieved great success\cite{deL1,deL2,deL3}.

Considering the fundamental importance of the classical theory of operator semigroups, a recent development of interest lay in extension of the classical theory of operator semigroups to random normed modules (briefly, $RN$ modules). The notion of an $RN$ module is a random generalization of that of a normed space. Thanks to Guo and Gigli's independent pioneering contributions \cite{B,C,Guopost,Guorelations,Guo20241}, the theory of $RN$ modules has obtained a systematic and deep development for the past thirty years, and has been successfully applied to several important fields such as random equations \cite{GuoZhang20,Skorohod,LM,LZD}, conditional risk measures \cite{Guoprogress} and nonsmooth differential geometry on metric measure spaces \cite{B,C,Guo2024,E}. The theory of operator semigroups on $RN$ modules began with the literature \cite{Guozhang}, where in order to establish the Stone representation theorem of a semigroup of strongly continuous unitary operators on complex random inner product modules,  Guo and Zhang demonstrated the fundamental theorem of calculus for an $L^{0}$-Lipschitz function from a finite real closed interval to a complete $RN$ module. In fact, such a  fundamental theorem of calculus plays an important role in the study of the theory of operator semigroups on $RN$ modules. For example, using such a fundamental theorem of calculus, in 2012, Zhang studied some mean ergodic semigroups of random linear operators in \cite{Zhangjapan}, and in 2013, Zhang and Liu further investigated some almost surely (briefly, a.s.) bounded semigroups of linear operators in \cite{Z-L}. Subsequently, in 2019, Thang, Son and Thinh first established the Hille-Yosida generation theorem for contraction $C_0$-semigroups of continuous module homomorphisms in \cite{Thang}. In 2020, Zhang, Liu and Guo further established the Hille-Yosida generation theorem for a.s. bounded $C_0$-semigroups of continuous module homomorphisms by a different way in \cite{Z-L2019}. Also in 2020, Zhang et al. investigated the abstract Cauchy problems with respect to $C_0$-semigroups in complete $RN$ modules in \cite{ZhangCauchy}. In 2024, Son, Thang and Oanh first studied the exponentially bounded $C$-semigroups and their associated abstract Cauchy problems in complete $R N$ modules \cite{Son2023}. The purpose of this paper is to continue to study the theory of operator semigroups in complete $R N$ modules.

Clearly, an exponentially bounded  $C$-semigroup in a complete $R N$ module is locally a.s. bounded. However, due to Son, Thang and Oanh's work in 2024, a locally a.s. bounded $C$-semigroup in a complete $R N$ module may not be exponentially bounded.
Moreover, we know that an ordinary $C$-semigroup in a Banach space is automatically locally bounded. Thus it is necessary to study the locally a.s. bounded $C$-semigroup in a complete $R N$ module and such a locally a.s. bounded $C$-semigroup is also the natural generalization of an ordinary $C$-semigroup in a Banach space.
The main difficulty of this paper is that we are forced to work out new techniques to give several properties peculiar to locally a.s. bounded $C$-semigroups in a complete $R N$ module, which is not involved in the classical $C$-semigroup theory. Besides, the discussion with respect to the locally $L^{0}$-Lipschitz property is also a difficult point of this paper.

\par
This paper contains six sections: in Section 2, we will present some preliminaries; in Section 3, we will introduce the notion of a (mild) $C$-existence family in complete $RN$ modules and present some basic results on such $C$-existence families; in Section 4, we will give several important properties peculiar to locally  a.s. bounded $C$-semigroups in complete $RN$ modules; then, based on the above work, in Section 5, some relations among $C$-existence families, $C$-semigroups and their associated abstract Cauchy problems in the random setting are established; in Section 6, we will give an application to a type of stochastic differential equations.

\section{Preliminaries}\label{sec2}
Throughout this paper, we always assume that $(\Omega,\mathcal F,P)$ denotes a given probability space, $K$ the scalar field $R$ of real numbers or $\mathbb{C}$ of complex numbers, $N$ the set of natural numbers.
Moreover, $L^{0}(\mathcal F,K)$ denotes
the algebra of equivalence classes of $K$-valued $\mathcal F$-measurable random variables on $\Omega$ and $\bar{L}^{0}(\mathcal F,R)$ the set of equivalence classes of generalized real-valued $\mathcal F$-measurable random variables on $\Omega$. It is known from \cite{Dunford} that $\bar{L}^{0}(\mathcal F,R)$ is a complete lattice under the partial order $\leq:$ $f \leq g$ if and only if $f^0(\omega) \leq g^0(\omega)$ for almost all $\omega$ in $\Omega$, where $f^0$ and $g^0$ are arbitrarily chosen representatives of $f$ and $g$ in $\bar{L}^{0}(\mathcal F,R)$. Besides, for each nonempty subset $G$ of $\bar{L}^{0}(\mathcal F,R)$, let $\bigvee G$ and $\bigwedge G$ denote the supremum and infimum of $G$, respectively, then there are two sequences $\{\xi_{n},~n\in N\}$ and ${\{\eta_{n},~n\in N}\}$ in $G$ such that $\bigvee\limits_{n\geq1}$ $\xi_{n}=\bigvee G$ and $\bigwedge\limits_{n\geq1}$ $\eta_{n}=\bigwedge G$. Further, $L^0(\mathcal{F}, R)$, as a sublattice of $\bar{L}^0(\mathcal{F}, R)$, is complete in the sense that every subset with an upper bound has a supremum.


 As usual, $I_A$ denotes the characteristic function of $A$ for any $A \in \mathcal{F}$ and $\tilde{I}_A$ denotes the equivalence class of $I_A$. Besides, for any $f$ and $g$ in $\bar{L}^0(\mathcal{F}, R)$, $f>g$ means $f \geq g$ and $f \neq g$, and for any $D \in \mathcal{F}$, $f>g$ on $D$ means  $f^0(\omega)>g^0(\omega)$ for almost all $\omega \in D$, where $f^0$ and $g^0$ are arbitrarily chosen representatives of $f$ and $g$, respectively. Let $A=\left\{\omega \in \Omega \mid f^0(\omega)>\right.$ $\left.g^0(\omega)\right\}$, then we always use $[f>g]$ for the equivalence class of $A$ and often write $I_{[f>g]}$ for $\tilde{I}_A$, one can also understand such notations as $I_{[f \leq g]}, I_{[f \neq g]}$ and $I_{[f=g]}$. In particular, we denote $L^{0}_{+}(\mathcal {F})=\{\xi\in L^{0}(\mathcal {F},R)\,|\,\xi\geq 0\}$ and $L^{0}_{++}(\mathcal {F})=\{\xi\in L^{0}(\mathcal{F},R)\,|\,\xi> 0 ~\textrm{on}~  \Omega \}.$

Due to \cite{Guopost}, an ordered pair $(X,\|\cdot\|)$ is called an $RN$ module over $K$ with base $(\Omega,{\mathcal F},P)$ if $X$ is an $L^{0}({\mathcal F},K)-$module and $\|\cdot\|$ is a mapping from $X$ to $L^{0}_{+}(\mathcal {F})$ such that the following three axioms hold.

($RN-$1)\, $\|\xi x\|=|\xi|\cdot\|x\|$ for any $\xi\in L^{0}({\mathcal F},K)$ and $x\in X$;

($RN-$2)\, $\|x+y\|\leq \|x\|+\|y\|$ for any $x,y\in X$;

($RN-$3)\, $\|x\|=0$ implies $x=\theta$ (the null in $X$), where the mapping $\|\cdot\|$ is called an
$L^{0}-$norm of $X$ and $\|z\|$ is called the
$L^{0}-$norm of $z$ in $X$.

\par
Following the idea of Schweizer and Sklar for probabilistic normed spaces\cite{SS},
a topology is usually introduced on an $RN$ module $(X,\|\cdot\|)$ as follows: for any $\varepsilon>0$ and $0<\lambda<1$, set $U_{\theta}(\varepsilon,\lambda)=\{z\in X~|~P\{\omega\in \Omega~|~\|z\|(\omega)<\varepsilon\}>\lambda\},$ then the family $\{U_{\theta}(\varepsilon,\lambda)~|~\varepsilon>0,~0<\lambda<1\}$ becomes a local base for some metrizable linear topology, which is called the $(\varepsilon,\lambda)$-topology induced by the $L^{0}-$norm on $X$. In the sequel of this paper, given an $RN$ module $(X,\|\cdot\|)$, we always assume that $(X,\|\cdot\|)$ is endowed with the above $(\varepsilon,\lambda)$-topology. Besides, observing that a sequence $\{z_{n},~n\in N\}$ in $X$ converges to some $z_{0}\in X$ in the $(\varepsilon,\lambda)$-topology if and only if the sequence $\{\|z_{n}-z_{0}\|,~n\in N\}$ in $L^{0}({\mathcal F},R)$ converges to 0 in probability $P$.
\par
Let $(X_{1},\|\cdot\|_{1})$ and $(X_{2},\|\cdot\|_{2})$ be two $RN$ modules over $K$ with base $(\Omega,{\mathcal F},P)$. A linear operator $T$ from $X_{1}$ to $X_{2}$ is said to be a.s. bounded if there is an $\eta\in L^{0}_{+}(\mathcal {F})$ satisfying $\|Tz\|_{2}\leq \eta\cdot\|z\|_{1}$ for any $z\in X_{1}$. Denote by $B(X_{1},X_{2})$ the $L^{0}({\mathcal F},K)-$module of a.s. bounded linear operators from $X_{1}$ to $X_{2}$, and define a mapping $\|\cdot\|:B(X_{1},X_{2})\rightarrow L^{0}_{+}(\mathcal {F})$ by $\|T\|:=\bigwedge\{\eta\in L^{0}_{+}(\mathcal{F})~|~\|Tz\|_{2}\leq \eta\cdot\|z\|_{1}~\textrm{for any } z\in X_{1}\}$ for any $T\in B(X_{1},X_{2})$, then one can obtain that $(B(X_{1},X_{2}),\|\cdot\|)$ is still an $RN$ module.

\par
Proposition 2.1 below shows the exact relation between an a.s. bounded linear operator and a continuous module homomorphism on an $RN$ module.

\begin{proposition}[\cite{Guorelations}]\label{proposition2.1}
Let $(X_{1},\|\cdot\|_{1})$ and $(X_{2},\|\cdot\|_{2})$ be two $RN$ modules over $K$ with base $(\Omega,{\mathcal F},P)$ and $T$ a linear operator from $X_{1}$ to $X_{2}$. Then the following statements hold.

(1) $T$ belongs to $B\left(X_1, X_2\right)$ if and only if $T$ is a continuous module homomorphism from $X_1$ to $X_2$;

(2) If $T$ belongs to $B\left(X_1, X_2\right)$, then $\|T\|=\bigvee\left\{\|T x\|_2: x \in X_1\right.$ and $\left.\|x\|_1 \leq 1\right\}$, where 1 stands for the unit element of $L^0(\mathcal{F}, R)$.

\end{proposition}
\par
Let $[s,t]$ be a finite closed real interval and $X$ an $RN$ module over $K$ with base $(\Omega,{\mathcal F},P)$. A function $f:[s,t]\rightarrow X$ is said to be $L^{0}$-Lipschitz on $[s,t]$ if there is an $\eta\in L^{0}_{+}(\mathcal{F})$ such that $\|f(s_{1})-f(s_{2})\|\leq \eta|s_{1}-s_{2}|$ for any $s_{1},s_{2}\in[s,t].$
Further, a function $f:[s,t]\rightarrow X$ is $L^{0}$-Lipschitz on $[s,t]$ if and only if $f$ satisfies the difference quotient a.s. bounded assumption, i.e.,
$\bigvee\{\|\frac{f(s_{1})-f(s_{2})}{s_{1}-s_{2}}\|~|~s_{1},~s_{2}\in[s,t]~\textrm{and}~s_{1}\neq s_{2}\}$ belongs to $L^{0}_{+}({\mathcal F}).$

\begin{proposition}[\cite{Guozhang}]\label{proposition2.2}
Suppose that $X$ is a complete $RN$ module and $f:[s,t]\rightarrow X$ is continuously differentiable. If $f$ is $L^{0}$-Lipschitz on $[s,t]$, then $f^{'}$ is Riemann integrable on $[s,t]$ and
$\int_{s}^{t}f^{'}(u)du=f(t)-f(s).$
\end{proposition}


It is well known that a continuous function from $[s,t]$ to a Banach space is automatically bounded, but a continuous function from $[s,t]$ to a complete $RN$ module $X$ may not be a.s. bounded. Fortunately, a sufficient condition for a continuous function from $[s,t]$ to $X$ to be Riemann integrable has been given, that is, if $f:[s,t]\rightarrow X$ is a continuous function satisfying that $\bigvee\limits_{u\in[s,t]}\ {\|f(u)\|}$ belongs to $L^{0}_{+}(\mathcal {F}),$ then $f$ is Riemann integrable. Based on this fact, Propositions 2.3 and 2.4 below hold.

\begin{proposition}[\cite{Guozhang}]\label{proposition2.3}
Suppose that $f$ is a continuous function from $[s,t]$ to a complete $RN$ module $X$ satisfying $\bigvee\limits_{u\in[s,t]}\ {\|f(u)\|} \in L_{+}^{0}(\mathcal {F}),$ then the following statements hold.

(1) $\|\int_s^t f(u) d u\| \leq \int_s^t\|f(u)\| d u$;

(2) Let $G(l)=\int_s^l f(u) d u$ for any $l \in[s, t]$, then $G$ is differentiable on $[s, t]$ and $G^{\prime}(l)=f(l)$.
\end{proposition}

\begin{proposition}[\cite{Z-L}]\label{proposition2.4} Suppose that $f$ is a continuous function from $[s,t]$ to $L^{0}(\mathcal F,R)$ satisfying $\bigvee\limits_{u\in[s,t]}\ {|f(u)|} \in L^1(\mathcal{F}, R),$ where $L^1(\mathcal{F}, R)=\{\xi\in L^0(\mathcal{F}, R)~|~\int_{\Omega}|\xi|dP<+\infty\}$, then $\int_{\Omega}[\int_{s}^{t}f(u)du]dP=\int_{s}^{t}[\int_{\Omega}f(u)dP]du.$
\end{proposition}

\section{$C$-existence families in complete $RN$ modules}\indent
In the sequel of this section, we always assume that $X$ is a complete $RN$ module, $A$ is a module homomorphism from $D(A)$ into $X$ and $[D(A)]$ stands for the $RN$ module $D(A)$ endowed with the graph $L^0$-norm $$\quad\|z\|_{[D(A)]} \equiv\|z\|+\|A z\| \quad(\forall z \in D(A)).$$
Besides, $C([0, \infty),[D(A)])$ denotes the set of continuous functions from $[0, \infty)$ to $[D(A)]$, $C([0, \infty), X)$ the set of continuous functions from $[0, \infty)$ to $X$, and $C^{1}([0, \infty), X)$ the set of continuously differentiable functions from $[0, \infty)$ to $X$.
\par
Given $z \in X$, the abstract Cauchy problem for $A$ with initial data $z$ consists of finding a solution $W(s, z)$ to the following problem
\begin{equation}
\left\{\begin{array}{l}
\frac{d W(s, z)}{d s}=A W(s, z), \quad \forall s>0, \\
W(0, z)=z .
\end{array}\right.
\end{equation}

(a) If a mapping $s \mapsto W(s, z) \in C([0, \infty),[D(A)]) \cap C^1([0, \infty), X)$ satisfies the abstract Cauchy problem (1), then we call it a solution of (1).

(b) If a mapping $s \mapsto W(s, z) \in C([0, \infty), X)$ satisfies that $\bigvee\limits_{u \in[0, s]}\|W(u, z)\|$ belongs to $L_{+}^0(\mathcal{F})$ for each $s \geq 0$ and $z \in X$, and further $v(s, z) \equiv$ $\int_0^s W(u, z) d u \in D(A)$ and $\frac{d v(s, z)}{d s}=A(v(s, z))+z$, i.e.,
$$
W(s, z)=A \int_0^s W(u, z) d u+z
$$
for any $s \geq 0$, then we call it a mild solution of (1). Obviously, $v \in$ $C([0, \infty),[D(A)])$.

\par
The family of continuous module homomorphisms $\{V(t): s\geq 0\}$ is said to be locally a.s. bounded if $\bigvee\limits_{s \in[0, l]}\|V(s)\|$ belongs to $L_{+}^0 (\mathcal {F})$ for any $l>0$.
\par
Throughout this section, $C$ is assumed to be an a.s. bounded linear operator on $X$.
\begin{definition}\label{definition4.1}
The locally a.s. bounded family of continuous module homomorphisms $\{V(s): s \geq 0\}\subseteq B(X)$ is called a mild $C$-existence family for $A$ if the following statements hold.

(1) The mapping $s \mapsto V(s) z$ from $[0, \infty)$ to $X$ is continuous for any $z \in X$;

(2) For any $z \in X$ and $s \geq 0$, we have $\int_0^s V(u) z d u \in D(A)$ and $A \int_0^s V(u) z d u$ $=V(s) z-C z$.
\end{definition}


\begin{definition}\label{definition4.2}
The locally a.s. bounded family of continuous module homomorphisms $\{V(s): s \geq 0\} \subseteq B([D(A)])$ is called a $C$-existence family for $A$ if
the following statements hold.

(1) The mapping $s \mapsto V(s) z$ from $[0, \infty)$ into $[D(A)]$ is continuous for any $z \in D(A)$;

(2) For any $z \in D(A)$ and $s \geq 0$, we have
$$
\int_0^s A V(u) z d u=V(s) z-C z .
$$
\end{definition}\label{definition4.3}
\begin{definition}
Suppose that $\{V(s):s\geq 0\}$ is a mild $C$-existence family for $A$. If the family of continuous module homomorphisms $\{V(s)|_{[D(A)]}: s\geq 0\}$
is also a $C$-existence family for $A$, then $\{V(s):s\geq 0\}$ is called a strong $C$-existence family for $A$.
\end{definition}

\par
Subsequently, the following Theorem 3.4 and Corollary 3.5 are devoted to studying when a mild $C$-existence family becomes a strong $C$-existence family in the random setting.

\begin{theorem}\label{theorem4.5}
Let $A$ be a closed module homomorphism, $\{V(s): s \geq 0\}$ a mild $C$-existence family for $A$, and for any $z\in X$, the mapping $s \mapsto V(s)z$ from $[0, \infty)$ into $X$ is differentiable at $s=s_0$. Then $V(s_0)z$ belongs to $D(A)$ and $A V(s_0)z=\left.\frac{d V(s)z}{d s}\right|_{s=s_0 }$.
\end{theorem}

\begin{proof}
Let $s_0 \geq 0$ be fixed, since the family $\{V(s): s \geq 0\}$ is locally a.s. bounded, we have $\bigvee\limits_{u \in[0, s_0]}\|V(u)\| \in L_{+}^0(\mathcal {F})$ and $\bigvee\limits_{u \in[0, s_0+1]}\|V(u)\| \in L_{+}^0(\mathcal {F})$.
Thus $\bigvee\limits_{u \in[s_0, s_0+1]}\|V(u)\| \in L_{+}^0(\mathcal {F}).$
Set $$\xi_{s_0}=\bigvee\limits_{u \in[s_0, s_0+1]}\left\|V(u)z-V\left(s_0\right)z\right\|,$$ then clearly $\xi_{s_0} \in L_{+}^0(\mathcal {F})$.
Take $A_{n,s_0}=\left[n-1\leq \xi_{s_0}<n\right]$ for each $n \in N$, then $A_{n, s_0} \in L_{+}^0(\mathcal {F})$,
$A_{i, s_0} \cap A_{j, s_0}=\varnothing$ for any $i, j \in N$ and $i \neq j$, and further $\sum\limits_{n=1}^{\infty} A_{n, s_0}=\Omega$.
Besides, for each $n \in N, \int_{A_{n, s_0}}\left\|V(u)z-V\left(s_0\right)z\right\| d P$ is continuous with respect to $u$ on $\left[s_0, s_0+1\right]$. In fact, it is sufficient to prove that
$
\int_{A_{n, s_0}}\left\|V(u)z-V\left(s_0\right)z\right\|  d P$
converges to 0 as $u\downarrow 0$ for each $n \in N$. Since, for each $n \in N$,

$$I_{A_{n, s_0}}\left\|V(u)z-V\left(s_0\right)z\right\| \leq n$$
for any $u\in [s_0,s_0+1]$ and $$I_{A_{n, s_0}}\|V(u)z-V(s_0)z\| \rightarrow 0$$ in  probability $P$ as $u\downarrow s_0$, according to Lebesgue's dominated convergence theorem, we have $\int_{A _{n, s_0}}\left\|V(u)z-V\left(s_0\right)z\right\| d P \rightarrow 0$ as $u \downarrow s_0$. For any  $k \in N $, let $z_k \equiv k \cdot \int_{s_0}^{s_0+\frac{1}{k}} V(u)z d u$, then, due to
Proposition 2.4, we have
$$
\begin{aligned}
& \int_{\Omega}\left\|I_{A_{n, s_0}}z_k-I_{A_{n, s_0}} V\left(s_0\right)z\right\| d P \\
& =\int_{\Omega}\left\|k \int_{s_0}^{s_0+\frac{1}{k}} I_{A_{n, s_0}} V(u)z d u-k \int_{s_0}^{s_0+\frac{1}{k}} I_{A_{n, s_0}} V\left(s_0\right)z d u\right\| d P \\
& =k \int_{\Omega}\left\|\int_{s_0}^{s_0+\frac{1}{k}} I_{A_{n, s_0}}\left(V(u)z-V\left(s_0\right)z\right) d u\right\| d P \\
& \leq k \int_{\Omega}\left[\int_{s_0}^{s_0+\frac{1}{k}}\left\|I_{A_{n, s_0}}\left(V(u)z-V\left(s_0\right)z\right)\right\| d u\right] d P \\
& =k \int_{s_0}^{s_0+\frac{1}{k}}\left[\int_{\Omega} \| I_{A_{n, s_0}}V(u)z-V\left(s_0\right)z\| d P\right] d u \\
&\leq\max_{u \in\left[s_0, s_0+\frac{1}{k}\right]}\left\{\int_ {A_{n, s_0}}\left\|V(u) z-V\left(s_0\right) z\right\| d P\right\} \\
& \rightarrow 0 \text { as } k \rightarrow \infty
\end{aligned}
$$
for each $n \in N$,
which shows that $\left\|I_{A_{n, s_0}} z_k-I_{A_{n, s_0}} V\left(s_0\right) z\right\|$ converges to 0 in probability $P$ as $k \rightarrow \infty$ for each $n \in N$. Since $\sum\limits_{n=1}^{\infty} P\left(A_{n, s_0}\right)=P\left(\sum\limits_{n=1}^{\infty} A_{n, s_0}\right)$ $=P(\Omega)=1$, we have $\sum\limits_{n=1}^{\infty}I_{A_{n, s_0}} \cdot z_k$ converges to $\sum\limits_{n=1}^{\infty}I_{A_{n, s_0}} V\left(s_0\right) z$ in the $(\varepsilon, \lambda)$-topology as $k \rightarrow \infty$, i.e., $z_k$ converges to $V(s_0)z$ in the $(\varepsilon, \lambda)$-topology as $k \rightarrow \infty$. Since $\{V(s): s \geq 0\}$ is a mild  $C$-existence family for $A$, one can obtain $A z_k=k\left(V( s_0+\frac{1}{k}\right) z$ -- $V(s_0)z$) for each $k \in N$, and clearly  $A z_k$ converges to $\frac{d V(s)z}{d s}|_{s=s_0}$ as $k \rightarrow \infty$. Further, since $A$ is closed, we have $V(s_0)z\in D(A)$ and $A V(s_0)z=\frac{d V(s)z}{d s}|_{s=s_0}$.
\end{proof}

\begin{corollary}\label{corollary4.6}
Let $A$ be a closed module homomorphism, $\{V(s): s \geq 0\}$ a mild $C$-existence family for $A$. If one of the following statements holds:

(1) the mapping $s \mapsto V(s) z$ belongs to $C^1([0, \infty), X)$  for any $z \in X$ and the family $\left\{\left.V(s)\right|_{[D(A)]}: s \geq 0\right\}$ is locally a.s. bounded;

(2) the mapping $s \mapsto V(s) z$ belongs to $C([0, \infty),[D(A)])$ for any $z \in D(A)$ and the family $\left\{\left.V(s)\right|_{[D(A)]}: s \geq 0\right\}$ is locally a.s. bounded;

(3) $V(s) A$ is contained in $A V(s)$ for any $s \geq 0$,
 then $\{V(s): s \geq 0\}$ is a strong $C$-existence family for $A$.
\end{corollary}
\begin{proof}
 For (1), according to Theorem 3.4, one can obtain that $V(s)z \in D(A)$ and $A V(u)z=\frac{d V(u)z}{d u}$ for any $u \geq 0$. Thus the mapping $u \mapsto A V(u)z$ from $[0,\infty)$ to $X$ is continuous, which implies that the mapping $s \mapsto V(s) z$ from $[0, \infty)$ into $[D(A)]$ is continuous for any $z \in D(A)$. Since $A$ is closed and the family $\left\{\left.V(s)\right|_{[D(A)]}: s \geq 0\right\}$ is locally a.s. bounded, it follows that $A\left(\int_0^s V(u)z d u\right)=\int_0^s A V(u)zd u$ for any $z\in D(A)$ , which shows that $\{V(s): s \geq 0\}$ is a strong $C$-existence family for $A$.

For (2), for any $z\in D(A)$, we have $V(s)z \in D(A)$ and $\|V(s)z\|_{[D(A)]}=\| V(s)z\|+\| A V(s)z\|.$ Thus the mapping $s \mapsto A V(s)z$ from $[0,\infty)$ to $X$ is continuous and locally a.s. bounded since the family $\left\{\left.V(s)\right|_{[D(A)]}: s \geq 0\right\}$ is locally a.s. bounded. Consequently, $A\left(\int_0^s V(u)z d u\right)=\int_0^s A V(u)zd u$ for any $z\in D(A)$ since $A$ is closed, which shows that $\{V(s): s \geq 0\}$ is a strong $C$-existence family for $A$.

For $(3)$, it is sufficient to note that hypothesis
$(3)$ implies hypothesis $(2)$.
\end{proof}

\section{$C$-semigroups in complete $RN$ modules}\label{section3}
The main results of this section are Theorems 4.3 and 4.4. The focus is Lemma 4.1, which plays a crucial role in the proof of Theorem 4.3.

\begin{definition}\label{definition3.1}\cite{Son2023}
Let $X$ be a complete $RN$ module, $B(X)$ the set of continuous module homomorphisms on $X$
and $C\in B(X)$ an injective operator on $X$. Then a family $\{V(s):s\geq0\}\subset B(X)$ is called a $C$-semigroup on $X$ if

(1) $V(s)$ is strongly continuous, this is, for any $z \in X$, the mapping $s \rightarrow$ $V(s) z$ from $[0, \infty)$ into $X$ is continuous;

(2) $V(0)=C$;

(3) $C V(s+t)=V(s) V(t)$ for any $s, t \geq 0$.

\end{definition}


\begin{definition}\label{definition3.4}
Let $\{V(s):s\geq0\}$ be a $C$-semigroup on an $RN$ module $X$, and denote by $R(C)$ the range of $C$. Define
$$D(A)=\{z\in X: \lim_{s\downarrow0}\frac{V(s)z-Cz}{s}~exists~and~belongs~to~R(C)\}$$
and
$$Az=C^{-1}\lim_{s\downarrow0}\frac{V(s)z-Cz}{s}$$
for any $z\in D(A)$, then the mapping $A:D(A)\rightarrow X$ is called the generator of $\{V(s):s\geq0\}$, also denoted by $(A,D(A))$ in this paper.
\end{definition}
\par
Let $X$ be an $RN$ module over $K$ with base $(\Omega,{\mathcal F},P)$, it is clear that, for any  $C$-semigroup $\{V(s): s \geq 0\}$,
 $\bigvee\limits_{s \in[0, l]}\left\|V(s) z\right\|$ may not be in $L^{0}_{+}(\mathcal {F}$) for any $z\in X$ and $l>0$. In this paper,
a $C$-semigroup $\{V(s): s \geq 0\}$ is said to be locally a.s. bounded if for any $l>0$, $\bigvee\limits_{s \in[0, l]}\left\|V(s) \right\|$ is in $L_{+}^0(\mathcal {F})$. Besides, $\{V({s}): s \geq 0\}$ is said to be exponentially bounded if there are $\tau \in  L^{0}({\mathcal {F}},R)$ and $W \in L_{+}^0(\mathcal {F})$ satisfying $\|V(s)\|\leq W e^{\tau s}$ for any $s \geq 0$. Furthermore, a function $g$ from $[0, \infty)$ to $X$ is said to be locally $L^{0}$-Lipschitz if for any $l>0$, there is a $\xi_l \in L_{+}^0(\mathcal {F})$ satisfying $\left\|g\left(s_1\right)-g\left(s_2\right)\right\| \leq \xi_l\left|s_1-s_2\right|$ for any $s_1, s_2 \in[0, l]$.

\par
Now, set
$$A_{s}=\frac{V(s)-C}{s}$$
for any $s>0$ and $(A,D(A))$ denotes the generator of $\{V(s): s\geq0\}$, i.e., $Az=C^{-1}\lim\limits_{s\downarrow0}A_{s}z$ for any $z\in D(A)$.
\par
Motivated by the work of \cite{Z-L}, we give Lemma 4.1 below, one will see
that for any $l>0$, $\bigvee\limits_{s\in (0, l]}\|CA_{s}z\|$ is a.s. bounded for any $z\in D(A)$. It is
such a special property of $\{V(s): s \geq 0\}$ that makes some following proofs possible.
In the sequel of this section, we always assume that $(X,\|\cdot\|)$ is a complete $RN$ module.

\begin{lemma}\label{lemma3.5}
Suppose that $\{V(s):s \geq 0\}$ is a locally a.s. bounded $C$-semigroup on $X$ with the generator $(A, D(A))$.
Then, for any  $l>0$, $\bigvee\limits_{s\in (0, l]}\|CA_{s}z\|$ belongs to $L^{0}_{+}({\mathcal {F}})$
for any $z \in D(A)$.
\end{lemma}

\begin{proof}
For any $s>0$, set
$$
\eta(s, z)=\bigvee_{t \in [0, s]}\|V(t) A_s z\|
$$
for any $z \in D(A)$, then $\|CA_{s}z\|\leq \eta(s, z)$.

Further, we have
\begin{align}
& \eta(s, z)\nonumber \\\nonumber
& =\bigvee_{t \in [0, s]}\left\|V(t) \frac{V(s)z-Cz}{s}\right\|\\\nonumber
&=\bigvee_{t \in [0, s]}\left\|\frac{CV(t+s)z-CV(t)z}{s}\right\| \\\nonumber
& =\bigvee_{t \in [0, s]}\left\|\frac{V\left(t+\frac{s}{2}\right) V\left(\frac{s}{2}\right)z-CV\left(t+\frac{s}{2}\right)z+V(t) V\left(\frac{s}{2}\right)z-CV(t)z}{s}\right\| \\\nonumber
& \leq \bigvee_{t \in [0,s]}\left\|\frac{V\left(t+\frac{s}{2}\right)\left(V\left(\frac{s}{2}\right)z-Cz\right)}{2\left(\frac{s}{2}\right)}\right\|+\bigvee_{t \in [0, s]}\left\|\frac{V(t)\left(V\left(\frac{s}{2}\right) z-Cz\right)}{2\left(\frac{s}{2}\right)}\right\| \\\nonumber
& =\bigvee_{t \in [0, s]}\left\|\frac{V\left(t+\frac{s}{2}\right) A_\frac{s}{2} z}{2}\right\|+\bigvee_{t \in [0, s]}\left\|\frac{V(t) A_\frac{s}{2} z}{2}\right\| \\\nonumber
& \leq \bigvee_{t \in [0, \frac{3s}{2}]}\|V(t)A_\frac{s}{2} z\|\\\nonumber
& \leq \bigvee_{t \in [0, \frac{3s}{2}+\frac{s}{4}]}\|V(t)A_\frac{s}{4} z\|\\\nonumber
& \leq \cdot\cdot\cdot\\\nonumber
& \leq \bigvee_{t \in [0, 2s]}\|V(t)A_\frac{s}{2^n} z\|\\\nonumber
& \leq \bigvee_{t \in [0, 2s]}\|V(t)\|\|A_\frac{s}{2^n} z\|\\\nonumber
\end{align}
for any $s > 0$, $n\in N$ and $z\in D(A)$. Thus
$$\|CA_{s}z\|\leq \eta(s, z)\leq \bigvee_{t \in [0, 2s]}\|V(t)\|\|A_\frac{s}{2^n} z\|$$
for any $s > 0$, $n\in N$ and $z\in D(A)$. Letting $n\rightarrow\infty$ in the above inequality, one can obtain that
$$\|CA_{s}z\|\leq \bigvee_{t \in [0, 2s]}\|V(t)\|\|CAz\|$$
for any $s > 0$ and $z\in D(A)$.

Consequently, for any  $l>0$,
$$
\begin{aligned}
\bigvee_{s\in (0, l]}\|CA_{s}z\|&\leq \bigvee_{s\in (0, l]}\bigvee_{t \in [0, 2s]}\|V(t)\|\|CAz\|\\
&\leq \bigvee_{t \in [0, 2l]}\|V(t)\|\|CAz\|\\
&\in L^{0}_{+}({\mathcal {F}})
\end{aligned}
$$
for any $z\in D(A)$, which completes the proof of Lemma 4.1.
\end{proof}

\par
Based on Lemma 4.1, we can state Theorem 4.3 below, which will be used in the proof of Theorem 4.4.

\begin{theorem}\label{theorem3.6}
Suppose that $\{V(s):s \geq 0\}$ is a locally a.s. bounded $C$-semigroup on $X$ with the generator $(A, D(A))$. For any $z \in C(D(A))$, define a mapping $g{:}~ [0,+\infty) \rightarrow X$ by $g(s)=CV(s) z$. Then $g$ is locally $L^0$-Lipschitz.
\end{theorem}
\begin{proof}
Since $\{V(s):s \geq 0\}$ is a locally a.s. bounded $C$-semigroup, we have that $\bigvee_{s \in[0, l]}\left\|V(s) \right\|$ is in $L_{+}^0(\mathcal {F})$ for any $l>0$.
Thus, according to Lemma 4.1, we have
$$
\bigvee\{\|\frac{CV(l_{1})z-CV(l_{2})z}{l_{1}-l_{2}}\|~|~l_{1},~l_{2}\in[0,l]~\textrm{and}~l_{1}\neq l_{2}\}\in L^{0}_{+}({\mathcal{F}})
$$
for any $z\in C(D(A))$ and $l>0$. In fact, for any $z\in C(D(A))$, there is a $y\in D(A)$ such that $z=Cy$, then

$$
\begin{aligned}
&\bigvee\{\|\frac{CV(l_{1})z-CV(l_{2})z}{l_{1}-l_{2}}\|~|~l_{1},~l_{2}\in [0, l]~\textrm{and}~l_{1}\neq l_{2}\}\\
\qquad=&\bigvee\{\|\frac{C^{2}V(l_{1})y-C^{2}V(l_{2})y}{l_{1}-l_{2}}\|~|~l_{1},~l_{2}\in [0, l]~\textrm{and}~l_{1}\neq l_{2}\}\\
\qquad=&\bigvee\{\|\frac{CV(l_{2})\cdot[ V(l_{1}-l_{2})y-Cy]}{l_{1}-l_{2}}\|~|~l_{1},~l_{2}\in [0, l]~\textrm{and}~l_{1}\neq l_{2}\}\\
\qquad\leq &\bigvee_{s \in[0, l]}\left\|V(s) \right\|\cdot\bigvee\{\|\frac{C[ V(l_{1}-l_{2})y-Cy]}{l_{1}-l_{2}}\|~|~l_{1},~l_{2}\in [0, l]~\textrm{and}~l_{1}\neq l_{2}\}\\
\in &L^{0}_{+}({\mathcal{F}})
\end{aligned}
$$
for any $l>0$, which shows that $g$ is locally $L^{0}$-Lipschitz.
\end{proof}

\begin{lemma}\label{lemma3.7}
Suppose that $\{V(s):s\geq0\}$ is a locally a.s. bounded $C$-semigroup on $X$ and its generator is denoted by $(A, D(A))$. If a function $f:[0, \infty) \rightarrow X$ is continuously differentiable and locally $L^0$--Lipschitz, then

(a) $\lim\limits_{s \downarrow 0} \frac{1}{s} \int_0^s V(u) f(u) d u=C f(0)$;

(b) $\int_0^s V(u) f(u) d u \in D(A)$ for any $s \geq 0$ and
$$
A \int_0^s V(u) f(u) d u=V(s) f(s)-C f(0)-\int_0^s V(u) f^{\prime}(u) d u .
$$
\end{lemma}

\begin{proof}
 (a) Let $l>0$ be fixed, since $f:~[0, \infty)\rightarrow X$ is locally $L^0$-Lipschitz, we have that $\bigvee\limits_{u \in[0, l]}\|f(u)\|$ belongs to $L_{+}^0(\mathcal {F})$. Further, since $\{V(s): s \geq 0\}$ is locally a.s. bounded, we have that $\bigvee\limits_{u \in[0, l]}\|V(u)\|$ belongs to $L_{+}^0(\mathcal {F})$.
Thus $\bigvee\limits_{u \in[0, l]}\|V(u) f(u)\|$ belongs to $L_{+}^0(\mathcal {F})$. Set $\xi_l=\bigvee\limits_{u \in[0, l]}\|V(u) f(u)-Cf(0)\|$, then clearly $\xi_l \in L_{+}^{0}(\mathcal {F})$. Take $A_{n, l}=\left[n-1 \leq \xi_l<n\right]$ for each $n \in N$, then one can obtain that $A_{n, l} \in \mathcal {F}$, $A_{i ,l} \cap A_{j ,l}=\emptyset$ for any $i, j \in N$ and $i \neq j$, and further $\sum\limits_{n=1}^{\infty} A_{n, l}=\Omega$.
It is clear that for each $n \in N$,
$\int_{A_{n, l}}\|V(u) f(u)-Cf(0)\| d P$ is continuous with respect to $u$ on $[0, l]$.
In fact, it is sufficient to prove that
$\int_{A_{n, l}}\|V(u) f(u)-Cf(0)\| d P$ converges to 0 as $u \downarrow 0$ for each $n \in N$.
Since, for each $n \in N$,
$I_{A_{n, l}}\|V(u) f(u)-C f(0)\| \leq n$ for any $u\in [0,l]$ and $I_{A_{n, l}}\|V(u) f(u)-Cf(0)\| \rightarrow 0$ in probability $P$ as $u \downarrow 0$, due to Lebesgue's dominated convergence theorem, we have $\int_{A_{n, l}}\|V(u)f(u)-C f(0)\| d P \rightarrow 0$ as $u \downarrow 0$. Consequently, for any $s\in (0, l]$, according to Proposition $2.4$, we have
$$
\begin{aligned}
& \int_{\Omega}\left[\frac{1}{s} \int_0^s\left\|V(u) f(u) I_{A_{n, l}}-C f(0) I_{A_{n, l}}\right\| d u\right] d P \\
&= \frac{1}{s} \int_0^s\left[\int_{A_{n, l}}\|V(u) f(u)-Cf(0)\| d P\right] d u \\
&\leq\max _{u \in\left[0, s\right]}\left[\int_{A_{n, l}}\|V(u) f(u)-Cf(0)\| d P\right]\\
&\rightarrow  0 \text { as $s$} \downarrow 0
\end{aligned}
$$
for each $n \in N$.
Thus $\frac{1}{s} \int_0^s\left\|V(u) f(u) I_{A _{n, l}}-C f(0) I_{A _{n, l}}\right\| d u$ converges to 0 in probability $P$ as $s \downarrow 0$ for each $n \in N$.
Moreover,
$$
\begin{aligned}
& \left\|\frac{1}{s} \int_0^s V(u) f(u) I_{A _{n, l}} d u-C f(0) I_{A _{n, l}}\right\| \\
&=  \left\| \frac{1}{s} \int_0^s [V(u) f(u) I_{A _{n, l}} -Cf(0) I_{A _{n, l}}] d u \right\| \\
&\leq  \frac{1}{s} \int_0^s\left\|V(u) f(u) I_{A _{n, l}}-Cf(0) I_{A _{n, l}}\right\| d u
\end{aligned}
$$
for any $s\in (0, l]$ and $n \in N$, which implies that $\left\|\frac{1}{s} \int_0^s V(u) f(u) I_{A _{n, l}}du-C f(0) I_{A_{n, l}}\right\|$ also converges to 0 in probability $P$ as $s \downarrow 0$. Since $\sum\limits_{n=1}^{\infty} P\left(A_{n, l}\right)=P\left(\sum\limits_{n=1}^{\infty} A _{n, l}\right)=P(\Omega)=1$,
it follows that $\frac{1}{s} \int_0^s V(u) f(u) d u$ converges to $Cf(0)$ in the $(\varepsilon, \lambda)$-topology as $s \downarrow 0$, i.e., $$\lim _{s\downarrow0} \frac{1}{s} \int_0^s V(u) f(u) d u=Cf(0).$$

(b) Let $s>0$ be fixed and $z_s=\int_0^s V(u) f(u) d u$.
For any $h>0$, we have
$$
\begin{aligned}
\frac{1}{h}\left(V(h) z_s-C z_s\right) & =\frac{1}{h}\left[\int_0^s(C V(u+h)-CV(u)) f(u) d u\right] \\
& =\frac{1}{h}\left[\int_h^{s+h} CV(r) f(r-h) d r-\int_0^s CV(u) f(u) d u\right] \\
& =\frac{1}{h} \int_h^s CV(u)(f(u-h)-f(u)) d u  +\frac{1}{h} \int_s^{s+h} V(r) f(r-h) d r\\
&\quad-\frac{1}{h} \int_0^h CV(u) f(u) d u.
\end{aligned}
$$
Clearly, due to (a), one has
\begin{equation}
 \lim _{h \downarrow 0} \frac{1}{h} \int_0^h CV(u)f(u) d u=C^2 f(0).
\end{equation}
Next, we will prove that $$\lim _{h \downarrow 0} \frac{1}{h} \int_h^s CV(u)(f(u-h)-f(u)) d u=\int_0^s[-CV(u) f^{\prime}(u)]d u.$$ Since $f: [0, \infty) \rightarrow X$ is locally $L^0$-Lipschitz, it follows that
$$
\bigvee_{u \in[0, s]}\left\|f^{\prime}(u)\right\| \in L_{+}^0(\mathcal {F}) $$  and  $$\bigvee_{\substack{u \in[0, s] \\ h \in(0, u]}}\left\|\frac{f(u-h)-f(u)}{h}\right\| \in L_{+}^0(\mathcal {F}).
$$

Set $\eta_s=\bigvee\limits_{\substack{u \in[0, s] \\ h \in(0, u]}} \|CV(u) \frac{f(u-h)-f(u)}{h}+CV(u) f^{\prime}(u) \|$, then
$\eta_s \in L_{+}^0(\mathcal {F})$ since the $C$-semigroup $\{V(s):s \geq 0\}$ is locally a.s. bounded.
Take $A_{n, s}=\left[n-1 \leq \eta_s<n\right]$ for each $n \in N$, then $A_{n, s} \in \mathcal {F}$,
$A_{i, s} \cap A_{j, s}=\emptyset$ for any $i, j \in N$ and $i \neq j$, and further $\sum\limits_{n=1}^{\infty} A_{n, s}=\Omega$.
Clearly, $CV(u) \frac{f(u-h)-f(u)}{h}$ converges to $-C V(u)f^{\prime}(u)$ in the $(\varepsilon, \lambda)$-topology as $h \downarrow 0$ for any $u \in[0, s]$. Thus $\|I_{A_{n, s}}[C V(u) \frac{f(u-h)-f(u)}{h}+CV(u) f^{\prime}(u)]\|$ converges to $0$ in probability $P$ as $h \downarrow 0$ for any $u \in[0, s]$ and $n\in N$.
Since $$\left\|I_{A_{n, s}}\left[C V(u) \frac{f(u-h)-f(u)}{h}+CV(u) f^{\prime}(u)\right]\right\| \leq n$$
for any $u \in[0, s], ~h \in(0, u]$ and $n\in N$, due to
Lebesgue's dominated convergence theorem, we have
$$\int_{A_{n, s}} \left[C V(u) \frac{f(u-h)-f(u)}{h}+CV(u) f^{\prime}(u)\right]dP$$ converges to 0 in probability $P$ as $h{\downarrow 0}$.
Thus, for any $h \in(0, u]$, according to Proposition $2.4$,
$$
\begin{aligned}
& \int_{\Omega}\left[\int_0^s\|I_{A_{n, s}}[C V(u) \frac{f(u-h)-f(u)}{h}+CV(u) f^{\prime}(u)]\| d u\right] d P \\
&=   \int_0^s\left[\int_{A_{n, s}}\|C V(u) \frac{f(u-h)-f(u)}{h}+CV(u) f^{\prime}(u)\| d P\right] d u \\
&\rightarrow  0 \text { as $h$} \downarrow 0
\end{aligned}
$$
for each $n \in N$,
i.e., $\int_0^s\|I_{A_{n, s}}C V(u) \frac{f(u-h)-f(u)}{h}+I_{A_{n, s}}CV(u) f^{\prime}(u)\| d u$ converges to 0 in probability $P$ as $h \downarrow 0$.
Moreover, for any $h \in(0, u]$ and $n \in N$,
$$
\begin{aligned}
& \left\| \int_0^s I_{A_{n, s}}C V(u) \frac{f(u-h)-f(u)}{h}+I_{A_{n, s}}CV(u) f^{\prime}(u)du\right\| \\
&\leq  \int_0^s\left\|I_{A_{n, s}}C V(u) \frac{f(u-h)-f(u)}{h}+I_{A_{n, s}}CV(u) f^{\prime}(u)\right\| d u,
\end{aligned}
$$
which implies that $\|\int_0^s I_{A_{n, s}}C V(u) \frac{f(u-h)-f(u)}{h}+I_{A_{n, s}}CV(u) f^{\prime}(u)du\|$ also converges to 0 in probability $P$ as $h \downarrow 0$.

Since $\sum\limits_{n=1}^{\infty} P\left(A_{n, s}\right)=P\left(\sum\limits_{n=1}^{\infty} A_{n, s}\right)=P(\Omega)=1$, we have $$\left\|\int_0^s\left[C V(u) \frac{f(u-h)-f(u)}{h}+C V(u) f^{\prime}(u)\right] d u\right\|$$ converges to 0 in probability $P$ as $h \downarrow 0$, i.e.,
\begin{equation}
\lim _{h \downarrow 0} \frac{1}{h} \int_0^s CV(u)(f(u-h)-f(u)) d u=\int_0^s[-CV(u) f^{\prime}(u) ]d u.
\end{equation}
Thus
$$
\lim _{h \downarrow 0} \frac{1}{h} \int_h^s CV(u)(f(u-h)-f(u)) d u=\int_0^s[-CV(u) f^{\prime}(u)]d u
$$
since
$$
\lim _{h \downarrow 0} \frac{1}{h} \int_0^h CV(u)(f(u-h)-f(u)) d u=C^2f(0)-C^2f(0)=0.
$$
Similarly, one can prove that
\begin{equation}
\lim _{h \downarrow 0} \frac{1}{h} \int_s^{s+h} V(r) f(r-h) d r=CV(s) f(s).
\end{equation}
Consequently, $\lim\limits_{h \downarrow 0}\left(V(h) z_s-Cz_s\right)$ exists, combining (2), (3) and
(4), we have
$$
\lim _{h \downarrow 0}\left(V(h) z_s-C z_s\right)=\int_0^s-C V(u) f^{\prime}(u) d u+CV(s) f(s)-C^2 f(0).
$$
Thus $z_s \in D(A)$ and
$$
C A z_s=-\int_0^s CV(u) f^{\prime}(u) d u+CV(s) f(s)-C^2 f(0) \text {, }
$$
i.e.,
$$
A\left[\int_0^s V(u) f(u)du\right]=V(s) f(s)-C f(0)-\int_0^s V(u) f^{\prime}(u) d u \text {.}
$$

\end{proof}

\par
Based on Lemma 4.2 and Theorem 4.3, we can state Theorem 4.4 below, which generalizes the corresponding results of Son, Thang and Oanh's work in 2024.

\begin{theorem}\label{theorem3.7}
Let $\{V(s):~s\geq0\}$ be a locally a.s. bounded $C$-semigroup on $X$ with the generator $(A,D(A))$. Then

(1) $\lim\limits_{s \downarrow 0} \frac{1}{s} \int_0^s V(u) z d u=C z$ for any $z \in X$;

(2) for any $z \in X$ and $s \geq 0, \int_0^s V(u) z d u \in D(A)$ and
$$
A \int_0^s V(u) z d u=V(s) z-C z ;
$$

(3) $R(C) \subseteq \overline{D(A)}$;

(4) for any $z \in D(A)$ and $s \geq 0$, we have $V(s) z \in D(A)$ and
$$
\frac{d V(s) z}{d s}=V(s) A z=A V(s) z;
$$

(5) for any $z \in D(A)$ and $s \geq 0$, we have
$$
\int_0^s V(u) A z d u=V(s) z-C z;
$$

(6) $A$ is a closed module homomorphism and satisfies $C^{-1}AC=A$.

\end{theorem}

\begin{proof}
For any $z \in X$, let $f(s) \equiv z$ for any $s \geq 0$, then it is obvious that $f$ is continuously differentiable and locally $L^{0}$-Lipschitz. Thus statements (1) and (2) hold according to Lemma 4.2. Now we will prove $R(C) \subseteq \overline{D(A)}$. In fact, since for any $z \in X,~ \lim\limits_{s \downarrow 0} \frac{1}{s} \int_0^s V(u) z d u=Cz$ according to $(1)$, it follows from (2) that (3) clearly holds.

For (4), let $z\in D(A)$ and $s\geq 0$ be fixed,
since $V(s) \in B(X)$, it follows that
$$
\begin{aligned}
& V(s)CAz\\
&=V(s) \lim_{u\downarrow 0}[\frac{1}{u}[V(u)z-Cz]\\
&=\lim_{u\downarrow 0}\frac{1}{u}[V(u)(V(s)z)-C(V(s)z)] \\
&=CAV(s)z,\\
\end{aligned}
$$
which shows that $V(s)z\in D(A)$ and $V(s)Az=AV(s)z$.

Besides, observing that $$\frac{1}{u}[C(V(u+s)z)-C(V(s)z)]=\frac{1}{u}[V(u)V(s)z-CV(s)z]$$
for any $u>0$,
it follows that $\frac{d[CV(u)z]}{du}|_{u=s}$ exists and equals $CV(s)Az$, i.e., $$
\frac{d [CV(s) z]}{d s}=CV(s) A z=CA V(s) z.
$$
Thus statement (4) holds since $C$ is a.s. bounded and injective.

For (5), due to (4), one can obtain
$$\frac{d [CV(s) (Cz)]}{d s}=CV(s) A (Cz)=CA V(s) (Cz)
$$
for any $z\in D(A)$ and $s\geq 0$. Further, according to Theorem 4.3, the mapping $s\mapsto CV(s) (Cz)$ is
locally $L^0$-Lipschitz. Thus, due to Proposition 2.2, we have$$
\int_0^s C^{2}V(u) A z d u=C^{2}V(s) z-C^{3} z
$$ for any $z \in D(A)$ and $s\geq 0$, which shows that statement (5) holds since $C$ is a.s. bounded and injective.

For (6), it is obvious that $A$ is a module homomorphism. Next, we will prove that $A$ is closed. Suppose that there is a sequence $\left\{z_n,n\in N\right\} \subseteq D(A)$ such that $z_n \rightarrow z$ and $A z_n \rightarrow y$ as $n \rightarrow \infty$. Hence the sequence $\left\{\left\|A z_n-y\right\|, n \in {N}\right\}$ converges to 0 in probability $P$ as $n \rightarrow \infty$. Thus, according to Riesz theorem, there is a subsequence $\left\{\left\|Az_{n_k}-y\right\|, k \in N\right\}$ such that $\left\|Az_{n_k}-y\right\| \stackrel{a.s.}{\longrightarrow} 0$ as $k \rightarrow \infty$.
Since  $\{V(s): s \geq 0\}$ is locally a.s. bounded and
$$
\begin{aligned}
& \left\|\int_0^s V(u)A z_{n_k} d u-\int_0^s V(u) y d u\right\| \\
& \leq \int_0^s\left\|V(u)A z_{n_k}-V(u) y\right\| d u \\
&\leq \bigvee_{u \in[0, s]}\|V(u)\| \int_0^s\left\|Az_{n_k}-y\right\| d u \\
&= \bigvee_{u \in[0, s]}\|V(u)\| \left\|Az_{n_k}-y\right\|\cdot s\\
\end{aligned}
$$
for any $s\geq 0$,
it follows that $\left\|\int_0^s V(u)A z_{n_k} d u-\int_0^s V(u) y d u\right\| \stackrel{a.s.}{\longrightarrow} 0$ as $k \rightarrow \infty$.
Further, according to (5), we have
$$
\begin{aligned}
V(s) z-C z & =\lim _{k \rightarrow \infty}\left(V(s) z_{n_k}-C z_{n_k}\right) \\
& =\lim _{k \rightarrow \infty} \int_0^s V(u) A z_{n_k} d u \\
& =\int_0^s V(u) y d u,
\end{aligned}
$$
which shows that $\lim\limits_{s \downarrow 0} \frac{1}{s}(V(s) z-Cz)$ exists and equals $C y$.
Thus $z \in D(A)$ and $A z=y$, i.e., $A$ is a closed module homomorphism.

Due to (4), we have $A \subset C^{-1} A C$. Conversely, let $x \in D\left(C^{-1} A C\right)$, that is, $C x \in D(A)$ and $A C x \in R(C)$. By (5), we obtain that
$$
\begin{aligned}
C(V(t)x-C x)&=V(t) C x-C^2 x\\
&=\int_0^t V(s) A C x d s\\
&=C \int_0^t V(s) C^{-1} A C x d s
\end{aligned}
$$
for any  $t\geq 0$ and $x \in D\left(C^{-1} A C\right)$.
Consequently, for any  $t> 0$ and $x \in D\left(C^{-1} A C\right)$, one has
$$
\frac{V(t) x-C x}{t}=\frac{1}{t} \int_0^t V(s) C^{-1} A C x d s.
$$
Letting $t \rightarrow 0^{+}$ in the above equality, we have $A x=C^{-1} A C x$, which implies that $C^{-1} A C \subset A $.
\end{proof}

\section{Relations among $C$--existence families, $C$--semigroups and their associated abstract Cauchy problems in complete $RN$ modules}\label{sec5}

This section is devoted to giving some relations among $C$-existence families, $C$-semigroups and their associated abstract Cauchy problems in complete $RN$ modules. Clearly, it follows from Theorem 4.4 that a locally a.s. bounded $C$-semigroup $\{V(s): s \geq 0\}$ with its generator $(A, D(A))$ becomes a strong $C$-existence family for $(A, D(A))$. Thus Theorem 5.1 below guarantees the existence of solutions and mild solutions of the abstract Cauchy problem $(1)$. However, for a $C$-semigroup,
Theorem $5.2$ below shows that we can also guarantee the uniqueness of solutions and mild solutions of (1) even if such a $C$-semigroup is generated by an extension of $(A, D(A))$.

\begin{theorem}\label{theorem4.4}
(a) Let $\{V(s): s \geq 0\}$ be a mild $C$-existence family for a module homomorphism $A$. Then, for any $z \in R(C)$, there is a mild solution of the abstract Cauchy problem (1).

Further, for any $z_n \in X$ and $z_n \rightarrow 0(n \rightarrow \infty), W\left(s, C z_n\right) \equiv V(s) z_n$, acted as a sequence of mild solutions of (1), converges to 0 uniformly on any compact subsets of $[0, \infty)$ as $z_n \rightarrow 0$.

(b) Let $\{V(s): s \geq 0\}$ be a $C$-existence family for a module homomorphism $A$. Then, for any $z \in C(D(A))$, there is a solution of the abstract Cauchy problem (1).

Further, both $W\left(s, C z_n\right) \equiv V(s) z_n \rightarrow 0$ and $A W\left(s, C z_n\right) \rightarrow 0$ uniformly on any compact subsets of $[0, \infty)$, whenever $A z_n \rightarrow 0$ and $z_n \rightarrow 0$.
\end{theorem}
\begin{proof}
(a) Since $\{V(s): s \geq 0\}$ is a mild $C$-existence family for $A$, we have that for any $z \in R(C)$, there is a $y \in X$ such that $z=Cy$, and further $W(s, Cy) \equiv V(s)y$ is a mild solution of $(1)$. Besides, since the family $\{V(s): s \geq 0\}$ is locally a.s. bounded, i.e., $\bigvee\limits_{s \in[0, l]}\|V(s)\| \in L_{+}^0(\mathcal {F})$ for any $l>0$, the analogue of well-posedness in (a) follows.

(b) Since $\{V(s): s \geq 0\}$ is a $C$-existence family for $A$, we have that for any $z \in C(D(A))$, there is a $y \in D(A)$ such that $z=Cy$, and further $W(s, Cy) \equiv V(s)y$ is a solution of $(1)$. Since $\{V(s):s\geq 0\}\subseteq B([D(A)])$ is locally a.s. bounded, i.e., $\bigvee\limits_{s \in[0, l]}\|V(s)\|_{[D(A)]} \in L_{+}^0(\mathcal {F})$ for any $l>0$, which implies that the analogue of well-posedness in (b) follows.
\end{proof}

\begin{theorem}\label{theorem5.1}
(a) Let $\{V(s): s \geq 0\}$ be a locally a.s. bounded $C$-semigroup on $X$ with its generator $(A, D(A))$. Then $\{V(s): s \geq 0\}$ is a strong $C$-existence family for $(A, D(A))$.

(b) If a locally a.s. bounded $C$-semigroup is generated by an extension of $(A, D(A))$, then all solutions and mild solutions satisfying the locally $L^0$-Lipschitz condition of the abstract Cauchy problem (1) are unique.
\end{theorem}

\begin{proof}
(a) It is clear from Theorem 4.4.

(b) Suppose that $\tilde{A}$ is the generator of the $C$-semigroup $\{V(s): s \geq 0\}$ and $A \subset \tilde{A}$. Suppose that there is a locally $L^0$-Lipschitz function $v:[0, \infty) \rightarrow X$ such that $v^{\prime}(s)=A(v(s))$ for any $s \geq 0$ and $v(0)=0$. Let $s \geq 0$ be fixed, define a function $f: [0, s] \rightarrow X$ by $f(u)=C^{2}V(s-u) v(u)$ for any $0 \leq u \leq s$, then
$$
\begin{aligned}
f^{\prime}(u) & =C^{2}\cdot\frac{d V(s-u) v(u)}{d u} \\
& =C^{2}\cdot[V(s-u) Av(u)-\tilde{A} V(s-u) v(u)] \\
& =C^{2}\cdot[V(s-u) Av(u)-V(s-u) \tilde{A} v(u)]\\
&=0.
\end{aligned}
$$
Further, according to Lemma 4.1, we have
\begin{align}\nonumber
& \bigvee\left\{\left\|\frac{f\left(u_1\right)-f\left(u_2\right)}{u_1-u_2}\right\| ~|~ u_1, u_2 \in[0, s] \text { and } u_1 \neq u_2\right\} \\\nonumber
& =\bigvee\left\{\left\| \frac{C^{2}V\left(s-u_1\right) v\left(u_1\right)-C^{2}V\left(s-u_2\right) v\left(u_2\right)}{u_1-u_2} \right\| ~|~ u_1, u_2 \in[0, s] \text { and } u_1 \neq u_2\right \} \\\nonumber
& =\bigvee \{\| \frac{C^{2}V(s-u_1) v(u_1)-C^{2}V(s-u_1) v(u_2)}{u_1-u_2}\\\nonumber
&~~~~~~~+\frac{C^{2}V(s-u_1) v(u_2)-C^{2}V(s-u_2) v(u_2)}{{u_1-u_2}} \left\|\right.|~u_1, u_2 \in[0, s] \text { and } u_1 \neq u_2\} \\\nonumber
& \leq \bigvee_{u \in[0, s]}\|V(u)\|\cdot \|C^{2}\| \cdot \bigvee\left\{\left\|\frac{v\left(u_1\right)-v\left(u_2\right)}{u_1-u_2}\right\| ~|~ u_1, u_2 \in[0, s] \text { and } u_1 \neq u_2\right\} \\\nonumber
& ~~~~~~~+\bigvee\left\{\| \frac{CV(s-u_2)[ V(u_2-u_1)v(u_2)-Cv(u_2)]}{u_2-u_1}\|~|~  u_1, u_2 \in[0, s] \text { and } u_1 \neq u_2\right \} \\\nonumber
& \leq \bigvee_{u \in[0, s]}\|V(u)\|\cdot \|C^{2}\| \cdot\bigvee\left\{\left\|\frac{v\left(u_1\right)-v\left(u_2\right)}{u_1-u_2}\right\| ~|~ u_1, u_2 \in[0, s] \text { and } u_1 \neq u_2\right\} \\\nonumber
&~~~~~~~+\bigvee_{u \in[0, s]}\|V(u)\| \cdot \bigvee\left\{\left\|\frac{V(u) (Cv(u_2))-V(0)(Cv(u_2))}{u}\right\| ~|~ u\in [0,s] \text { and } u_2 \in[0, s]\right\} \\\nonumber
&\in L_{+}^0(\mathcal {F})\nonumber
\end{align}
since the function $v$ is locally $L^{0}$-Lipschitz. Thus, according to Proposition 2.2, we have that
$$
\int_0^s f^{\prime}(u) d u=f(s)-f(0) \text {}
$$
for any $s \geq 0$,
i.e., $C^{3} v(s)=C^2V(s) v(0)=0$, thus $v(s) \equiv 0$ since $C$ is injective, which shows that all solutions and mild solutions satisfying the locally $L^0$-Lipschitz condition of the abstract Cauchy problem $(1)$ are unique.
\end{proof}


\begin{remark}\label{remark5.10}
If we choose ${\mathcal F}=\{\Omega,\emptyset\}$, then a complete $RN$ module $X$ reduces to a Banach space $X$ and the locally a.s. bounded $C$-semigroup $\{V(s):s\geq0\}$ reduces to an ordinary $C$-semigroup on $X$. Thus the main results with respect to locally a.s. bounded families of operators in this paper also generalize the corresponding results in Banach spaces \cite[Chapter 2 and Chapter 3]{deL1} since a continuous function from $[0, \infty)$ to a Banach space $X$ is automatically locally bounded.
\end{remark}
\par
Based on Theorems 5.1 and 5.2, one can immediately obtain Corollaries 5.3 and 5.4 below.

\begin{corollary}[\cite{ZhangCauchy}]\label{corollary5.11}
Suppose that $\{V(s):s\geq0\}$ is a locally a.s. bounded $C_{0}$-semigroup on $X$ with the generator $(A,D(A))$. Then, under the locally $L^{0}$-Lipschitz condition on the solution, the abstract Cauchy problem (1)
has a unique  solution $u(s):=V(s)z$ belonging to $C([0, \infty),[D(A)]) \cap C^{1}([0, \infty), X)$ for any $z\in D(A)$.
\end{corollary}

\begin{corollary}[\cite{Son2023}]\label{corollary5.11}
Suppose that $\{V(s):s\geq0\}$ is an exponentially bounded $C$-semigroup on $X$ with the generator $(A,D(A))$. Then, under the locally $L^{0}$-Lipschitz condition on the solution, the abstract Cauchy problem (1) has a unique solution
$u(s):=V(s) C^{-1}z$ belonging to $C([0, \infty),[D(A)]) \cap C^{1}([0, \infty), X)$ for any $z\in C(D(A))$.
\end{corollary}

\section{An application to a type of stochastic differential equations}
\begin{example}\label{example5.12}

Let $(\Omega, \mathcal{F}, P)$ be a given  probability space. Let $\mathcal{B}^2(R^+)$ denote the space of all real-valued measurable processes $Y=\{Y_t\}_{t\geq 0}$ on $(\Omega, \mathcal{F}, P)$ satisfying
$$
P\left\{\omega \in\Omega \mid\int_0^\infty Y^2_t(\omega) d t<\infty \right\}=1,
$$
define a mapping $\|\cdot\|$: $\mathcal{B}^2(R^+)\rightarrow L_{+}^0(\mathcal{F})$ by
$$
\|Y\|=\|\{Y_t\}_{t\geq 0}\|=\left(\int_0^{\infty}Y^2_td t\right)^{\frac{1}{2}}
$$
for any $Y \in \mathcal{B}^2(R^+)$.
Here, we identify $Y$ with $\bar{Y}$ in $\mathcal{B}^2(R^+)$ if
$$
P\left\{\omega \in\Omega \mid\int_0^\infty (Y_t(\omega)-\bar{Y}_t(\omega))^2 d t=0 \right\}=1.
$$
Clearly, $\|\cdot\|$ defines an $L^0$-norm on $\mathcal{B}^2(R^+)$ and $(\mathcal{B}^2(R^+),\|\cdot\|)$ becomes a complete $RN$ module.

Let $H$ be an operator defined by
$$
H(Y)=\{I_{[0,1]}(t)tZY_t\}_{t\geq 0}
$$
for any $Y \in \mathcal{B}^2(R^+)$, where
$$
I_{[0, 1]}(t)= \begin{cases}1, & \text { if } t \in[0, 1], \\ 0, & \text { otherwise},\end{cases}
$$
and $Z$ is a standard Gaussian random variable.
Let $C=e^{H}$, then we have
$$
\begin{aligned}
C(Y)&=\{1+(H)+\frac{(H)^2}{2!}+\cdot\cdot\cdot)Y_t\}_{t\geq 0}\\
&=\{Y_t+(H)(Y_t)+\frac{(H)^2}{2!}(Y_t)+\cdot\cdot\cdot\}_{t\geq 0}\\
&=\{Y_t+(I_{[0,1]}(t)tZ)Y_t+\frac{(I_{[0,1]}(t)tZ)^2}{2!}Y_t+\cdot\cdot\cdot\}_{t\geq 0}\\
&=\{e^{I_{[0,1]}(t)tZ}Y_t\}_{t\geq 0}
\end{aligned}
$$
for any $Y \in \mathcal{B}^2(R^+)$.
Similarly, one has
$$
e^{ZHs}( Y)=\{e^{Z^2I_{[0,1]}(t)ts}Y_t\}_{t\geq 0}
$$ for any $Y \in \mathcal{B}^2(R^+)$
and $s \geq 0$. Let
$$
V(s)=e^{ZHs} e^{H}
$$
for any $s \geq 0$, then $\{V(s): s \geq 0\}$ is a $C$-semigroup on  $\mathcal{B}^2(R^+)$. It is easy to check that for any $s \geq 0$,
$$
\begin{aligned}
\|V(s)(Y)\|&=\|\{e^{Z^2I_{[0,1]}(t)ts+ZI_{[0,1]}(t)t}Y_t\}_{t\geq 0}\|\\
&=\left(\int_0^{\infty}(e^{Z^2I_{[0,1]}(t)ts+ZI_{[0,1]}(t)t}Y_t)^2d t\right)^{\frac{1}{2}}\\
&\leq e^{Z}e^{Z^2s}\left(\int_0^{\infty}Y^2_td t\right)^{\frac{1}{2}}\\
&= e^{Z}e^{Z^2s}\|Y\|
\end{aligned}
$$
for any $Y \in \mathcal{B}^2(R^+)$. Thus, we have

$$\|V(s)\|\leq e^{Z}e^{Z^2s}$$
for any $s \geq 0$.
Consequently, $\{V({s}): s \geq 0\}$ is exponentially bounded. Further, one can obtain that
$A=ZH$ is the infinitesimal generator of the exponentially bounded $C$-semigroup $\{V(s): s \geq 0\}$. Consider the following stochastic differential equation on $\mathcal{B}^2(R^+)$

\begin{equation}
\left\{\begin{array}{l}
\frac{dY_t}{d t}=AY_t,~~ \forall t\geq 0, \\
Y_0=y \in C(D(A)),
\end{array}\right.
\end{equation}
due to Theorems 5.1 and 5.2, under the locally $L^0$-Lipschitz condition of the solution, we have that the stochastic differential equation $(5)$
possesses a unique solution $Y_t:=V(t) C^{-1}y$.

\end{example}


\noindent{\bf Acknowledgments:}

{\noindent The} study was supported by the National Natural Science Foundation of China (Grant No. 12171361) and the Humanities and Social Science Youth Foundation of Ministry of Education (Grant No. 20YJC790174).

\begin{spacing}{2.0}
\begin{large}
\noindent
{\bf AUTHOR DECLARATIONS}
\end{large}

{\noindent\bf Conflict of Interest}

{\noindent The} author has no conflicts to disclose.

{\noindent\bf Author Contributions}

{\noindent\bf Xia Zhang:} Writing - original draft (equal). {\bf Leilei Wei:} Writing - original draft (equal). {\bf Ming Liu:} Writing - original draft (equal).
\end{spacing}

\begin{spacing}{2.0}
\begin{large}
\noindent
{\bf  DATA AVAILABILITY}
\end{large}

{\noindent Data} sharing is not applicable to this article as no new data were created or analyzed in this study.
\end{spacing}

\bibliographystyle{be}

\end{document}